\documentclass[paper=a4, fontsize=12pt]{article}
\usepackage[T1]{fontenc}

\usepackage[english]{babel}															
\usepackage{amssymb}	
\usepackage{amsmath}
\usepackage{amsfonts}
\usepackage{amsthm}
\usepackage{t1enc}
\usepackage{graphicx}
\usepackage{color}	
\usepackage{url}

\usepackage{sectsty}
\usepackage{cite}

\usepackage{fancyhdr}
\numberwithin{equation}{section}		
\numberwithin{figure}{section}			
\numberwithin{table}{section}

\newtheorem{Lemma}{Lemma}
\newtheorem{Corollary}{Corollary}

\newtheorem{Theorem}{Theorem}
\theoremstyle{definition}

\newtheorem{Example}{Example}
\newtheorem{Remark}{Remark}

\title{\textbf{Construction of Butson matrices using Fourier matrices as input}}
\author{
		\normalfont 								\normalsize
        Farouk Adda $^{*}$\\[-3pt]		\normalsize\\
}
\date{}

\begin{document}
\thispagestyle{empty}
\maketitle

\noindent\textbf{Abstract.} Butson matrices are square orthogonal matrices, denoted by $BH(m,n)$, whose entries are the complex $m$th roots of unity and satisfy the condition\\ $BH(m,n)\cdot{BH(m,n)}^*=nI_n$, where ${BH(m,n)}^*$ is the conjugate transpose of $BH(m,n)$ and $I_n$ is the identity matrix. In this work, we propose constructions for $BH(m,(n-1)n)$ then $BH(m,(\frac{n}{2}-1)n)$, when $n$ and $m$ are even numbers, using the existing $BH(m,n)$. For each case, we provide two construction methods: one uses a single input Butson matrix, and another uses two input Butson matrices. Moreover, we present some results about the construction of Hadamard matrices.\\

\let\thefootnote\relax\footnote{This research is supported by the Laboratory of Algebra and Numbers Theory (LATN), Algeria.\\$*$ Corresponding author. E-mail: adda-farouk@hotmail.com.(A.Farouk).\\ U.S.T.H.B., Faculty of mathematics, B.P.32, El Allia 16111 Bab Ezzouar, Algiers, Algeria.}
\\
\textbf{Keywords:} Fourier matrices ; Scarpis construction; Butson matrices.\\
\\
\textbf{AMS Subject Classification: 15B20, 05C31.} \\
\section{Introduction}
Fourier matrices are the matrices $F_n$ of order $n$, with entries defined by
$$[F_n]_{i\;j}=e^{\dfrac{2\pi\mathbf{i}(i-1)(j-1)}{n}},$$
where $\mathbf{i}^2=-1$. This matrix is orthogonal, and its entries are the complex
$n$th roots of unity, making it a $BH(n,n)$. This demonstrates that a Butson matrix exists for any order, unlike Hadamard matrices.\\
\\
Butson matrices are also known as Complex Hadamard Matrices. However, the latter term is typically reserved for matrices whose entries are from the set $\{\mathbf{i},-\mathbf{i},-1,1\}$, which are powers of the complex $4$th root of unity. These matrices share some properties with Hadamard matrices. They can be classified under row and column permutations, as well as row or column multiplication by a complex $m$th root of unity. Consequently, each class can be represented by a \textit{normalized matrix}, also known as a \textit{dephased} matrix, where the first row and column consist entirely of ones. The sum of the entries in any row or column of such matrix, except for the first row or column, is equal to zero.\\
 \\
Another key property of this representative is that by deleting the first row and column, we obtain a matrix whose any dot product of two rows or two columns is equal to $-1$. This resulting matrix is called the \textit{Core} of the Butson matrix.\\
\\
The construction of Butson matrices, remains a challenging and unresolved problem in contemporary mathematics. Despite numerous proposed construction methods aimed at determining matrices of various orders (see \cite{c4,const1,const2,const3,const4}), an exhaustive list or a recursive construction method for any specific root of unity has yet to be established.\\
\\
The classification of Butson matrices is also a major open problem in this field, presenting a complexity that surpasses even that of their construction. Currently, classes are fully determined only for orders less than $6$ (see\cite{const2}). Numerous studies have attempted to further classify these matrices (see \cite{const2,class1,class2,class3,class4}), but no complete lists have been provided.\\
\\
Scarpis construction is one of the earliest methods for constructing Hadamard matrices, notable for utilizing existing Hadamard matrices as input. Introduced by U. Scarpis in 1898, this construction employs Hadamard matrices of order $p+1$ to produce a Hadamard matrix of order $p(p+1)$, where $p$ is a prime number congruent to 3 modulo 4 (see \cite{h5}). This work was later generalized by D-\u{Z}. Djokovi\'c in \cite{h8}, who extended it to include powers of prime numbers congruent to 3 modulo 4.\\
\\
Building on these ideas, we succeeded in \cite{h1} in developing a new construction that inputs a Hadamard matrix of order $2(q+1)$ satisfying certain conditions to output a Hadamard matrix of order $2q(q+1)$, where $q$ is a power of a prime number congruent to 1 modulo 4. We demonstrated that Paley type II matrices meet these conditions and can serve as input matrices.\\
\\
In \cite{h2}, we further extended the range of output matrices obtainable from this constructions by analyzing the row indexing process. We defined a set of Latin squares, termed Latin Squares Eligible for Scarpis Construction (LSESC), and utilized them in the establishment methods. We defined constructions of Hadamard matrices of orders $n(n-1)$ and $n(\frac{n}{2}-1)$ from a Hadamard matrix of order $n$. Additionally we proposed a construction of Hadamard matrices of order $n(\frac{n}{k}-1)$ from such matrix of order $n$, where $k$ is multiple of four that divides $n$ into an even number.\\
\\
The work of Scarpis was recently generalized in \cite{c4} to develop constructions of Butson matrices. This generalization provides a method for constructing a Butson matrix of order $q(q+1)$ from a Butson matrix of order $q+1$, where $q$ is any power of a prime number.\\
\\
In this article, we begin by introducing the concept of LSESC. We explore their relationship with Mutually Orthogonal Latin Squares (MOLS) and we initiate the construction of Butson matrices in subsequent sections. In the second section, we utilize these findings to construct Butson matrices of order $n(n-1)$ from those of order $n$. We propose two construction methods: one that uses a single input Butson matrix and another that employs two input matrices. In the final section, we present a method for constructing Butson matrices of order $n(\frac{n}{2}-1)$ from existing matrices of order $n$. These inputs must satisfy certain conditions, which are met by some Fourier matrices. At the end of each of the last two sections, we compute the number of output matrices. These matrices can be inequivalent, suggesting that their classification may require further refinement.\\
\\
First, let's recall some definitions and notation.\\
\\
The $i$-th row of a matrix $A=(a_{ij})_{1 \leq i\leq n}^{1 \leq j\leq n}$ is denoted by $\mathbf{a}_i$. $A^{\top}$ denotes the transpose matrix of $A$, and $A^*=\overline{A}^{\top}$ is the conjugate transpose of $A$. We denote by $\mathbb{A}_n$ the set $\{1,...,n\}$. Let $\mathbf{J}_m$ be the $1\times m$ matrix (row vector) with all entries equal to $1$. $\textrm{O}_{m,p}$ is the zero matrix of size $m\times p$. Two $n\times n$ matrices $X,Y$ are \textit{disjoint} if there is no position where they are both nonzero.\\
\\
Two vectors of $\mathbb{C}^n$ are orthogonal if their inner product (or dot product) is 0 (i.e., two $1\times n$ vectors $\mathbf{x}=(x_i)$ and $\mathbf{y}=(y_i)$ are orthogonal if\\ $\langle\mathbf{x},\mathbf{y}\rangle:=\sum_{i=1}^{n}x_i\overline{y_i}= \mathbf{x}\overline{{\mathbf{y}}}^{\top}= 0$).\\
\\
The Kronecker (or tensor) product $X \otimes Y$ of two matrices $X = (x_{ij})$ and $Y$ is the block matrix $X\otimes Y=(x_{ij}Y )$.\\
\\
Tensors are multidimensional arrays, with vectors being first-order tensors and matrices being second-order tensors. Tensors of order higher than two are referred to as \textit{higher order tensors}. The set of all tensors of  order $l$ and of the same dimenssions $n_1\times n_2 \times ....\times n_l$ form a vector space, and denoted by $\mathbb{T}_{n_1,n_2,...,n_l}$.\\
\\
In this study we focus on third-order tensors. A \textit{cubic tensor} is a third-order tensor where all dimensions are equal to an integer $n$. The set of all cubic tensors is denoted by $\mathbb{T}n$. Using the notation from \cite{b3}, tensors are denoted by calligraphic letters like $\mathcal{X}, ;\mathcal{Y}, \ldots$. \textit{Slices}, or two-dimensional sections of a tensor, include horizontal slices $\mathbf{X}{i::}$, lateral slices $\mathbf{X}{:j:}$, and frontal slices $\mathbf{X}{::k}=\mathbf{X}_k$.
\section{Latin squares eligible for Scarpis construction }
A \textit{Latin square} of order $n$ is an $n\times n$ matrix with entries chosen from the $n$-set of \textit{symbols}, such that each symbol appears exactly once in each row and exactly once in each column. Here, the rows and columns of a Latin square of order $n$ are indexed by the elements of $\mathbb{A}_n$, or the elements of a group of order $n$.\\
 \\
 We denote the set of all Latin squares of order $n$ over a fixed symbols set by $\mathbb{L}_n$. If we permute the rows or columns of a Latin square, then the result is a Latin square; and if we permute the symbols set of a Latin square or replace the symbols set by another symbols set, then the result is a Latin square. Any Latin square obtained from another Latin square, $L$, by any combination of these operations is said to be \textit{isotopic} to $L$. Clearly, isotopy is an equivalence relation. Then, Latin squares are classified by \textit{isotopy}.\\
\\
A pair of Latin squares $L$ and $L'$, of the same order are said to be \textit{orthogonal} if for each $a$ in the symbols set of $L$ and each $b$ in the symbols set of $L'$, there exists a unique pair $i,j$ such that $l_{i\;j}=a$ and $l'_{i\;j}=b$. $N(n)$ denotes the maximum cardinality of a set of MOLS of order $n$.\\
\\
\noindent Two Latin squares $L$ and $L'$ of order $n$ are said to be eligible for Scarpis construction, or \textit{LSESC}, if and only if for every $i,i'\in \{1,...,n\}$, there exists a unique $j\in\{1,...,n\}$, such that $l_{ij}=l'_{i'j}$.The maximal number of pairwise LSESC squares of order $n$ is denoted by $N'(n)$.\\
\\
By taking a set of LSESC and, for each entry $l_{i\;j}=a$, exchanging $a$ and $i$, the result is a set of MOLS, and vice versa. These two sets are conjugate, and the notions are equivalent. Thus, it follows that $N'(n)=N(n)$.\\
\\
A longstanding conjecture suggests that a complete set of MOLS with $N(n)=n-1$ exists only if $n$ is a power of prime number. This conjecture has been verified for orders less than or equal to $6$ but remains open for higher orders (see \cite{b2} Chapter 1). The conjecture can be extended to LSESC. While in the next constructions we need sets LSESC of cardinality $n-1$ for each $n$. So, if the conjecture is true, the values these constructions will extend are powers of prime numbers.\\
\\
An example of a complete set of LSESC is the set
$$\mathcal{L}=\{L_b:b\in GF(q)-{0}\},$$
where $L_b$ is the $q\times q$ matrix with entries $l_{ij}=x_i+bx_j$, and $GF(q)=\{x_1,x_2,...,x_q\}$ is the Galois field of order $q$. These are used in the classical construction of Scarpis and form a complete set of MOLS (see \cite{b2}). However, we can also construct a complete set of LSESC that is not a set of MOLS, as demonstrated in Theorem 2.6 in \cite{h2}. This theorem proposes an algorithm for constructing a set of LSESCs of size $n-1$ of each order $n$, which can subsequently be adapted to generate complete sets of MOLS.\\
\\
We consider a function that associate to every Latin square a cubic tensor defined as follows
$$\begin{array}{ccccc}
F&:&\mathbb{L}_n&\rightarrow&\mathbb{T}_n\\
&&L&\mapsto&\mathcal{X}
\end{array},$$
where $\mathcal{X}$ is the cubic tensor of frontal slices
$$\mathbf{X}_k=\left[\begin{array}{c}\delta(l_{i\;k},j)\end{array}\right],$$
and
$$\delta(i,j)=\left\{\begin{array}{cc}1&\mbox{if}\;i=j\\0&\mbox{otherwise}\end{array}\right..$$
We define also the functions
$$\begin{array}{ccccc}
G_m&:&\mathbb{T}_n&\rightarrow&\mathbb{T}_{mn,mn,n}\\
&&\mathcal{X}&\mapsto&\mathcal{Y}
\end{array}$$
where, $m\in\mathbb{N}$, and $\mathcal{Y}$ is the third-order tensor of frontal slices
$$\mathbf{Y}_k=I_m\otimes \mathbf{X}_k.$$
\begin{Remark}
Let $L\in\mathbb{L}_n$, and let $F(L)=\mathcal{X}$. Then
 \begin{itemize}
 \item[(1)]Every frontal slice $\mathbf{X}_t$ is a permutation matrix such that, the nonzero entry in row $i$ is at the $l_{it}$-th position. So, the frontal slices form a set of disjoint permutation matrices;
 \item[(2)]The horizontal slices $\mathbf{X}_{t::}$ also form a set of disjoint permutation matrices, where every row $i$ contain the nonzero element in the $l_{ti}$-th position ;
 \item[(3)]Furthermore, the lateral slices $\mathbf{X}_{:t:}$ form a set  of disjoint permutation matrices, where
                     \begin{eqnarray}L=\sum_{i=1}^{n}i\mathbf{X}_{:i:}.\end{eqnarray}
\end{itemize}
Therefore, using a fixed symbols set, the notion of isotopic is equivalent in the image of $F$ to the permutations over cubic tensors.
\end{Remark}

\section{Construction of $BH(m,(n-1)n)$ from $BH(m,n)$}
Let $\mathcal{BH}(m,d)$ denote the set of all normalized Butson matrices of order $d$ ( not classified under row or column permutations). We define a map $\Phi$ that associates to each Butson matrix of order $n$ a Butson matrix of order $(n-1)n$. Formally, we have
$$\Phi:\mathcal{BH}(m,n)\to\mathcal{BH}(m,(n-1)n)$$
$$\;\;\;\;H\to\Phi(H).$$
For a given complete set of LSESC of order $n$, $\{L_1,L_2,...,L_{n-1}\}$, we call complete tensor set of LSESC, the set $\{F(L_1),F(L_2),...,F(L_{n-1})\}$. Consequently, we obtain the following result:
\begin{Theorem}[Construction theorem]
Let $H\in \mathcal{BH}(m,n)$ of a \textit{Core} $C$. Let $\mathbb{X}=\{\mathcal{X}^{(1)},...,\mathcal{X}^{(n-2)}\}$ be a complete tensor set of LSESC of order $n-1$. Let $B_0=H'\otimes \mathbf{J}_{n-1}$, with $H'$ the sub-matrix of $H$ obtained by deleting the first row, and
\begin{eqnarray}B=\left[\begin{array}{c|cccc}
{\mathbf{J}_{n-1}}^{\top}\otimes \mathbf{c}_1&C&C&...&C\\
{\mathbf{J}_{n-1}}^{\top}\otimes \mathbf{c}_2&\mathbf{X}^{(1)}_{1}C&\mathbf{X}^{(1)}_{2}C&...&\mathbf{X}^{(1)}_{n-1}C\\
{\mathbf{J}_{n-1}}^{\top}\otimes \mathbf{c}_3&\mathbf{X}^{(2)}_{1}C&\mathbf{X}^{(2)}_{2}C&...&\mathbf{X}^{(2)}_{n-1}C\\
\vdots & \vdots&\vdots &...&\vdots\\
{\mathbf{J}_{n-1}}^{\top}\otimes \mathbf{c}_{n-1}&\mathbf{X}^{(n-2)}_{1}C&\mathbf{X}^{(n-2)}_{2}C&...&\mathbf{X}^{(n-2)}_{n-1}C\\
\end{array}
\right].
\end{eqnarray}
Then, the matrix
$$\Phi_{\mathbb{X}}(H)=\left[\begin{array}{c}
   B_0\\
  B\\
  \end{array}
\right]$$
is a Butson matrix of order $n(n-1)$ with entries powers of the complex $m$th root of unity.
\end{Theorem}
\begin{proof}
We prove the theorem by showing the row orthogonality of $\Phi_{\mathbb{X}}(H)$, the columns orthogonality follows immediately.\\
Firstly, let $B$ be divided into $n-1$ block matrices $\Gamma_k$:\\
$$\Gamma_0=\left[\begin{array}{c|cccc}{\mathbf{J}_{n-1}}^{\top}\otimes \mathbf{c}_1&C&C&...&C\\\end{array}\right].$$
The $j$th row of this matrix is,
\begin{eqnarray}\left[\mathbf{c}_1|\;\mathbf{c}_{j}\;\mathbf{c}_{j}\;...\;\mathbf{c}_{j}\right].\end{eqnarray}
When $k\neq 0$
$$\Gamma_{k}=\left[\begin{array}{c|cccc}{\mathbf{J}_{n-1}}^{\top}\otimes
\mathbf{c}_{k+1}&\mathbf{X}^{(k)}_{1}C&\mathbf{X}^{(k)}_{2}C&...&\mathbf{X}^{(k)}_{n-1}C\\
\end{array}\right],$$
its $j$th row is,
\begin{eqnarray}\left[\mathbf{c}(k+1)|\;\mathbf{c}({l^{(k)}}_{j\;1})\;\;\mathbf{c}({l^{(k)}}_{j\;2})\;...\;\mathbf{c}({l^{(k)}}_{j\;m-1})\right],\end{eqnarray}
Now, we show four cases:
\begin{itemize}
\item[1] Two rows of $B_0$ are orthogonal, as the Kronecker product saves orthogonality.
\item[2] A row of $B_0$ and another of $\Gamma_k$: Firstly, let's consider the case when $k\neq0$. The scalar product of these two rows becomes:
    \begin{eqnarray}
    \mathbf{b_0}_t\overline{{\mathbf{\Gamma_k}_r}}^{\top} &=& h'_{t\;1}\overline{\sum_{i=1}^{n-1}{[\mathbf{c}(k+1)^{\top}]}_i}+\sum_{i=1}^{n-1}h'_{t\;(i+1)}\overline{\sum_{j=1}^{n-1}[\mathbf{c}({l^{(k)}}_{ri})^{\top}]_j} \nonumber \\
    &=& -\sum_{i=1}^{n}h'_{t\;i}=0. \nonumber
    \end{eqnarray}
    Here $[\mathbf{c}({l^{(k)}}_{ri})]_j$ are components of the row vector $\mathbf{c}({l^{(k)}}_{ri})$. The result is obtained as the sums $\overline{\sum_{j=1}^{n-1}[\mathbf{c}({l^{(k)}}_{ji})^{\top}]_j} $ are all equal to $-1$, we have considered the core rows. If we choose the rows of $\Gamma_0$, the result follows immediately in the same way.
\item[3] Now we consider two rows of $\Gamma_k$. When $k=0$, we obtain the following,
    \begin{eqnarray}
    \mathbf{\Gamma_0}_t\overline{{\mathbf{\Gamma_0}_s}}^{\top} &=& \mathbf{c}(1)\overline{{\mathbf{c}(1)}}^{\top}+\sum_{i=1}^{n-1}\mathbf{c}(t)\overline{{\mathbf{c}(s)}}^{\top}=(n-1)-(n-1)=0.\nonumber
    \end{eqnarray}
    We pass now to rows of other $\Gamma_k$s. The inner product is
    \begin{eqnarray}
    \mathbf{\Gamma_k}_t\overline{{\mathbf{\Gamma_k}_r}}^{\top} &=& \mathbf{c}(k+1)\overline{{\mathbf{c}(k+1)}}^{\top}+\sum_{i=1}^{n-1}\mathbf{c}({l^{(k)}}_{ti})\overline{{\mathbf{c}({l^{(k)}}_{ri})}}^{\top}=(n-1)-(n-1)=0.\nonumber
    \end{eqnarray}
    As $l^{(k)}_{ti}$ and $l^{(k)}_{ri}$ are two entries of the column $i$ of the Latin square $L^{(k)}$, they must be different.
\item[4] Finally, a row of $\Gamma_k$ and another from $\Gamma_s$: when one of the sub-matrices is $\Gamma_0$. Without  a loss of generality, we can choose $k=0$. We obtain,
    \begin{eqnarray}
    \mathbf{\Gamma_0}_t\overline{{\mathbf{\Gamma_s}_r}}^{\top} &=& \mathbf{c}(1)\overline{{\mathbf{c}(s+1)}}^{\top}+\sum_{i=1}^{n-1}\mathbf{c}(t)\overline{{\mathbf{c}({l^{(s)}}_{ri})}}^{\top}=-1-(n-2)+(n-1)=0.\nonumber
    \end{eqnarray}
    Here surely $l^{(s)}_{ri}=t$ for some unique $i$ and different otherwise. Hence the result.\\
    At the end, we consider two blocks different from $\Gamma_0$. In this case, we obtain the following formula,
    \begin{eqnarray}
    \mathbf{\Gamma_k}_t\overline{{\mathbf{\Gamma_s}_r}}^{\top} &=& \mathbf{c}(k+1)\overline{{\mathbf{c}(s+1)}}^{\top}+\sum_{i=1}^{n-1}\mathbf{c}({l^{(k)}}_{ti})\overline{{\mathbf{c}({l^{(s)}}_{ri})}}^{\top}=-1-(n-2)+(n-1)=0.\nonumber
    \end{eqnarray}

    As the set $\mathbb{X}$ is of LSESC, there exists exactly one position $i'$ where \\ $l^{(k)}_{ti'}=l^{(s)}_{ri'}$. Thus the result.

\end{itemize}
This finishes the proof of orthogonality, and the obtained matrix is\\ $BH(m,(n-1)n)$.
\end{proof}
\begin{Example}
Let's consider the Fourier matrix $F_3$, by taking $\mathbf{j}=e^{\dfrac{2\pi\mathbf{i}}{3}}$. Then, $\mathbf{j}^2=\overline{\mathbf{j}}$ and $\mathbf{j}^3=1$, and
$$F_3=\left[\begin{array}{ccc}
1&1&1\\
1&\mathbf{j}&\overline{\mathbf{j}}\\
1&\overline{\mathbf{j}}&\mathbf{j}
\end{array}\right].$$
Here, we identify an LSESC of order $2$ for implementation in the construction of $BH(3,6)$. The set containing the single Latin square of this order serves as the unique LSESC for this order.
$$\mathbb{S}=\left\{L=\left[\begin{array}{cc}
1&2\\
2&1
\end{array}\right]\right\}.$$
The corresponding set of tensors is combined of the tensor $\mathcal{X}=F(L)$, $\mathbb{X}=\{\mathcal{X}\}$. It is the tensor of frontal slices
$$X_1=\left[\begin{array}{cc}
1&0\\
0&1
\end{array}\right],\;\;X_2=\left[\begin{array}{cc}
0&1\\
1&0
\end{array}\right]$$
Therefore, we obtain,
$$\Phi_{\mathbb{X}}(F_3)=\left[\begin{array}{ccc}
\begin{array}{ccc}
1&1
\end{array}&\begin{array}{cc}
\mathbf{j}&\mathbf{j}
\end{array}&\begin{array}{cc}
\overline{\mathbf{j}}&\overline{\mathbf{j}}
\end{array}\\
\begin{array}{cc}
1&1
\end{array}&\begin{array}{cc}
\overline{\mathbf{j}}&\overline{\mathbf{j}}
\end{array}&\begin{array}{cc}
\mathbf{j}&\mathbf{j}
\end{array}\\
\begin{array}{cc}
\mathbf{j}&\overline{\mathbf{j}}\\
\mathbf{j}&\overline{\mathbf{j}}
\end{array}&\begin{array}{cc}
\mathbf{j}&\overline{\mathbf{j}}\\
\overline{\mathbf{j}}&\mathbf{j}
\end{array}&\begin{array}{cc}
\mathbf{j}&\overline{\mathbf{j}}\\
\overline{\mathbf{j}}&\mathbf{j}
\end{array}\\
\begin{array}{cc}
\overline{\mathbf{j}}&\mathbf{j}\\
\overline{\mathbf{j}}&\mathbf{j}
\end{array}&\begin{array}{cc}
\mathbf{j}&\overline{\mathbf{j}}\\
\overline{\mathbf{j}}&\mathbf{j}
\end{array}&\begin{array}{cc}
\overline{\mathbf{j}}&\mathbf{j}\\
\mathbf{j}&\overline{\mathbf{j}}
\end{array}
\end{array}\right].$$
After normalization we obtain,
$$\Phi_{\mathbb{X}}(F_3)=\left[\begin{array}{cccccc}
1&1&1&1&1&1\\
1&1&\mathbf{j}&\mathbf{j}&\overline{\mathbf{j}}&\overline{\mathbf{j}}\\
1&\mathbf{j}&\overline{\mathbf{j}}&1&\mathbf{j}&\overline{\mathbf{j}}\\
1&\mathbf{j}&1&\overline{\mathbf{j}}&\overline{\mathbf{j}}&\mathbf{j}\\
1&\overline{\mathbf{j}}&\mathbf{j}&\overline{\mathbf{j}}&\mathbf{j}&1\\
1&\overline{\mathbf{j}}&\overline{\mathbf{j}}&\mathbf{j}&1&\mathbf{j}
\end{array}\right].$$
This matrix belongs to the class $S^{(0)}_6$ as specified in \cite{const2}.
\end{Example}
\noindent This construction extend the number of matrices obtained by the generalization of Scarpis construction proposed in \cite{c4}, potentially encompassing more classes of matrices. Furthermore, this method can be further extended to include additional matrices by utilizing two Butson matrices as inputs. In other words, we define:
$$\Phi:\mathcal{BH}(m,n)\times \mathcal{BH}(m,n)\to\mathcal{BH}(m,(n-1)n)$$
$$\;\;\;\;(G,H)\to\Phi(G,H).$$
Then we conclude as follows.
\begin{Corollary}
Let $(G,H)\in \mathcal{BH}(m,n)\times \mathcal{BH}(m,n)$. Let $\mathbb{X}=\{\mathcal{X}^{(1)},...,\mathcal{X}^{(n-2)}\}$ be a complete tensor set of LSESC of order $n-1$. Then we define $$\Phi_{\mathbb{X}}(G,H)=\left[\begin{array}{c}
   B_0\\
  B\\
  \end{array}
\right],$$
with $B_0=G'\otimes \mathbf{J}_{n-1}$, where $G'$ is the sub-matrix of $G$ obtained by deleting its first row, and $B$ is defined as in (3.1) by taking $C$ as the core of $H$.
\end{Corollary}
\begin{proof}
The orthogonality of the rows of $B$ is immediate in the same way as in Theorem 1. The rows of $B_0$ are mutually orthogonal because of the properties of Kronecker product. It remains to show the orthogonality of a row chosen from $B$, and another from $B_0$. We consider the same decomposition of Theorem 1 of $B$ into the blocks $\Gamma_k$. Then, we have for $k\neq0$,
\begin{eqnarray}
    \mathbf{b_0}_t\overline{{\mathbf{\Gamma_k}_r}}^{\top} &=& g'_{t\;1}\overline{\sum_{i=1}^{n-1}{[\mathbf{c}(k+1)^{\top}]}_i}+\sum_{i=1}^{n-1}g'_{t\;(i+1)}\overline{\sum_{j=1}^{n-1}[\mathbf{c}({l^{(k)}}_{ri})^{\top}]_j} \nonumber \\
    &=& -\sum_{i=1}^{n}g'_{t\;i}=0, \nonumber
    \end{eqnarray}
and the same result for $k=0$. Therefore, $\Phi_{\mathbb{X}}(G,H)$ is a $BH(m,n(n-1))$.
\end{proof}
\noindent This construction can be applied to the creation of Hadamard matrices, and by considering $\Phi_{\mathbb{X}}$ applied on a couple of Hadamard matrices $(G,H)\in \mathcal{H}_n\times \mathcal{H}_n$, where $\mathcal{H}_n$ normalized Hadamard matrices of order $n$ ( not classified under row or column permutations), we obtain the following result that can be proven similarly to Corollary 1.
\begin{Corollary}
Let $(G,H)\in \mathcal{H}_n\times \mathcal{H}_n$. Then we obtain the Hadamard matrix $\Phi_{\mathbb{X}}(G,H)$ defined in the same way as in Corollary 1, for each complete tensor set of LSESC $\mathbb{X}$ of order $n-1$.
\end{Corollary}
\noindent Finally, let's collect the number of Butson matrices that can be obtained by this constructions. By $MOLS(n-1)$, we denote the family of all complete sets of MOLS of order $n-1$ not classified under isotopy.\\
\\
From \cite{c4}, additional matrices can be derived from each construction given in Theorem 1 or Corollary 1 by deleting a row $(x_1\;x_2\;...\;x_n)$ of $H$ or $G$, respectively, then defining $B_0$. Additionally, we take
\begin{eqnarray}B=\left[\begin{array}{c|cccc}
D_{x_1}\cdot{\mathbf{J}_{n-1}}^{\top}\otimes \mathbf{c}_1&D_{x_2}\cdot C&D_{x_3}\cdot C&...&D_{x_n}\cdot C\\
D_{x_1}\cdot{\mathbf{J}_{n-1}}^{\top}\otimes \mathbf{c}_2&D_{x_2}\cdot\mathbf{X}^{(1)}_{1}C&D_{x_3}\cdot\mathbf{X}^{(1)}_{2}C&...&D_{x_n}\cdot\mathbf{X}^{(1)}_{n-1}C\\
D_{x_1}\cdot{\mathbf{J}_{n-1}}^{\top}\otimes \mathbf{c}_3&D_{x_2}\cdot\mathbf{X}^{(2)}_{1}C&D_{x_3}\cdot\mathbf{X}^{(2)}_{2}C&...&D_{x_n}\cdot\mathbf{X}^{(2)}_{n-1}C\\
\vdots & \vdots&\vdots &...&\vdots\\
D_{x_1}\cdot{\mathbf{J}_{n-1}}^{\top}\otimes \mathbf{c}_{n-1}&D_{x_2}\cdot\mathbf{X}^{(n-2)}_{1}C&D_{x_3}\cdot\mathbf{X}^{(n-2)}_{2}C&...&D_{x_n}\cdot\mathbf{X}^{(n-2)}_{n-1}C\\
\end{array}
\right].
\end{eqnarray}
where $D_{x_i}=Diag_{n-1}(x_i,x_i,...,x_i)$, for each given $i$. Notice that (3.1) is the construction obtained by deleting the first row of ones in the establishment of $B_0$. In conclusion, we obtain $$|MOLS(n-1)|\cdot|\mathcal{BH}(m,n)|^2\cdot n$$
matrices for each $n$.\\
\\
The same analysis can be applied to Hadamard matrices to get
$$|MOLS(n-1)|\cdot|\mathcal{H}_n|^2\cdot n$$
matrices for each given $n$. However, this classification is not sufficiently precise, as, for instance, taking $n=4$, $\Phi$ will define more than $24$ matrices of order $12$ must be equivalent, as there is only one class of Hadamard matrices of this order.
\section{Construction of $BH(m,(\frac{n}{2}-1)n)$ from $BH(m,n)$, when $n$ even}
Throughout this section, we let $n$ and $m$ be even numbers. We define a subset $\mathcal{A}^{(1,2)}(m,n)$ of $\mathcal{BH}(m,n)$ consisting of matrices that verifies the following conditions:
\begin{itemize}
\item[\textbf{C1.}] There exists a row $(x_1\;x_2\;...\;x_n)$ of the matrix such that the row\\ $(x_1\;-x_2\;x_3\;...\;x_{n-1}\;-x_n)$ is also a row of the matrix. (You can verify easily that these two rows are orthogonal).
\item[\textbf{C2.}] They contain at least a row and a column combined only from $-1$s and $1$s, where the common component of this row and column is $-1$.
\end{itemize}
$\mathcal{A}^{(1)}(m,n)$ and $\mathcal{A}^{(2)}(m,n)$ are used to denote the subsets of $\mathcal{BH}(m,n)$ that satisfy the first and second conditions, respectively.
\begin{Remark}
We notice that from any matrix that verifies \textbf{C2}, we can obtain a matrix that verifies \textbf{C1} by performing column permutations if necessary, then taking $(x_1\;x_2\;...\;x_n)=(1\;1\;...\;1)$.
\end{Remark}
\noindent We verify now for which values of $n$ Fourier matrices are in the subset $\mathcal{A}^{(1,2)}(m,n)$. For $i=\frac{n}{2}+1$, the components of the corresponding row of this matrix are
$${[F_n]}_{\frac{n}{2}+1\;j}=e^{\dfrac{2\pi\mathbf{i}(\frac{n}{2})(j-1)}{n}}=e^{\pi\mathbf{i}(j-1)}=-1^{j-1}.$$
It is a row of only 1s and -1s. Taking now $j=\frac{n}{2}+1$, the corresponding column is of entries,
$${[F_n]}_{i\;\frac{n}{2}+1}=e^{\pi\mathbf{i}(i-1)}=-1^{i-1}.$$
Also a column of only 1s and -1s. You can easily notice that this is the only such row and column in $F_n$. So, the common component of this row and column is the entry,
$$[F_n]_{\frac{n}{2}+1\;\frac{n}{2}+1}=-1^{\frac{n}{2}}.$$
It is equal to $1$ if $\frac{n}{2}$ even and $-1$ otherwise.\\
On the other hand, taking a row $i_0< \frac{n}{2}$ of $F_n$, the row $i_0+\frac{n}{2}$ is of entries,
\begin{eqnarray}
[F_n]_{i_0+\frac{n}{2}\;j} &=& e^{\dfrac{2\pi\mathbf{i}(i_0+\frac{n}{2}-1)(j-1)}{n}} \nonumber\\
&=& e^{\dfrac{2\pi\mathbf{i}(i_0-1)(j-1)+(\frac{n}{2})(j-1)}{n}} \nonumber\\
&=& e^{\dfrac{2\pi\mathbf{i}(i_0-1)(j-1)}{n}}\cdot e^{\dfrac{2\pi\mathbf{i}\frac{n}{2}(j-1)}{n}}\nonumber\\
&=& e^{\dfrac{2\pi\mathbf{i}(i_0-1)(j-1)}{n}}\cdot-1^{j-1},
\end{eqnarray}
which shows that the row $\mathbf{F_n}_{i_0}$ and the row $\mathbf{F_n}_{i_0+\frac{n}{2}}$ verify \textbf{C1}. We therefore obtain the following lemma.
\begin{Lemma}.\;
\begin{itemize}
\item[1.] For all $n\geq 2$, $F_n\in \mathcal{A}^{(1)}(n,n)$.
\item[2.] If $n=2d$ where $d$ is an odd number. Then, $F_n\in \mathcal{A}^{(2)}(n,n)$.
\end{itemize}
\end{Lemma}
\noindent The elements of the subset $\mathcal{A}^{(2)}(m,n)$ have particular properties that make them useful as an input matrix for the next construction, we summarize this properties in the following lemma.
\begin{Lemma}
Let $H\in\mathcal{A}^{(2)}(m,n)$. Then, $H$ contains a $(n-2)\times(n-2)$ sub-matrix
$$T=\left[\begin{array}{c}C\\D\end{array}\right]$$
that verifies the following properties:
\begin{itemize}
\item[1.] The dot product of any two rows of $C$ is equal to $-2$, and the same for rows of $D$.
\item[2.] The inner product of a row of $C$ and another from $D$ is equal to $0$.
\item[3.] The sum of the first $\dfrac{n-2}{2}$ entries of $C$ is equal to $-1$, and the sum of the last $\dfrac{n-2}{2}$ entries is also equal to $-1$.
\item[4.] The sum of the first $\dfrac{n-2}{2}$ entries of $D$ is equal to $-1$, and the sum of the last $\dfrac{n-2}{2}$ entries is equal to $1$.
\end{itemize}
\end{Lemma}
\begin{proof}
Let $H'$ be the matrix obtained from $H$ by performing the following permutations:
\begin{itemize}
\item[1)] A column permutation that puts the $\{1,-1\}$ column that verifies \textbf{C2} as the second column.
\item[2)] A row permutation that puts the $\{1,-1\}$ row that verifies \textbf{C2} as the second row
\item[3)] A column permutation that transforms the second row obtained from 2) to the row $(1\;-1\;1\;...\;1\;-1\;...\;-1)$(the first entry is 1 and the second is $-1$, then $\frac{n-2}{2}$ 1s, and $\frac{n-2}{2}$ -1s).
\item[4)] A row permutation that transforms the second column obtained from 1) to the column $(1\;-1\;1\;...\;1\;-1\;...\;-1)$.
\end{itemize}
$T$ is the submatrix obtained by deleting the first two rows and columns of $H'$. It is clearly a submatrix of $H$. We can write $H'$ as follows:
$$H'=\left[\begin{array}{c|c}
\begin{array}{cc}
1&1\\
1&-1
\end{array}&\begin{array}{cccccc}
1&...&1&1&...&1\\
1&...&1&-1&...&-1
\end{array}\\
\hline
\begin{array}{cc}
1&1\\
\vdots&\vdots\\
1&1
\end{array}&\begin{array}{c}
C
\end{array}\\
\hline
\begin{array}{cc}
1&-1\\
\vdots&\vdots\\
1&-1
\end{array}&\begin{array}{c}
D
\end{array}
\end{array}\right]$$
where,
$$T=\left[\begin{array}{c}C\\D\end{array}\right]$$
$C$ and $D$ are $\frac{n-2}{2}\times n-2$ matrices.\\
The inner product of two rows of $H'$ that contain a row of $C$ is
\begin{eqnarray}
\mathbf{h'}_t\cdot\overline{\mathbf{h'}}_r^{\top}&=&1+1+\mathbf{c}_{t-2}\overline{\mathbf{c}}_{r-2}^{\top}=0 \nonumber\\
&\Rightarrow& \mathbf{c}_{t-2}\overline{\mathbf{c}}_{r-2}^{\top}=-2.
\end{eqnarray}
Similarly, for two rows that contain $D$
\begin{eqnarray}
\mathbf{h'}_t\cdot\overline{\mathbf{h'}}_r^{\top}&=&1+-1\cdot-1+\mathbf{d}_{t-(\frac{n-6}{2})}\overline{\mathbf{d}}_{r-(\frac{n-6}{2})}^{\top}=0 \nonumber\\
&\Rightarrow& \mathbf{d}_{t-(\frac{n-6}{2})}\overline{\mathbf{d}}_{r-(\frac{n-6}{2})}^{\top}=-2.
\end{eqnarray}
Now, by taking a row that contains a row of $C$ and another that contains a row of $D$, we get
\begin{eqnarray}
\mathbf{h'}_t\cdot\overline{\mathbf{h'}}_r^{\top}&=&1+1\cdot-1+\mathbf{c}_{t-2}\overline{\mathbf{d}}_{r-(\frac{n-6}{2})}^{\top}=0 \nonumber\\
&\Rightarrow& \mathbf{c}_{t-2}\overline{\mathbf{d}}_{r-(\frac{n-6}{2})}^{\top}=0.
\end{eqnarray}
Hence, 1. and 2. of the lemma.\\
For a row $2<t\leq\frac{n+2}{2}$ of $H'$, let's put $X=\sum_{j=1}^{\frac{n-2}{2}}c_{t-2\;j}$ and $Y=\sum_{j=\frac{n-2}{2}+1}^{n-2}c_{t-2\;j}$. The inner product of the row $\mathbf{h'}_t$ and the second row of $H'$ is
\begin{eqnarray}
\mathbf{h'}_t\cdot\overline{\mathbf{h'}}_2^{\top}&=&1+1\cdot-1+\sum_{j=1}^{\frac{n-2}{2}}c_{t-2\;j}-\sum_{j=\frac{n-2}{2}+1}^{n-2}c_{t-2\;j}=0 \nonumber\\
&\Rightarrow& X-Y=0. \nonumber
\end{eqnarray}
On the other hand, the sum of the row $\mathbf{h'}_t$ is equal to zero. So,
$$1+1+\sum_{j=1}^{\frac{n-2}{2}}c_{t-2\;j}+\sum_{j=\frac{n-2}{2}+1}^{n-2}c_{t-2\;j}=0\Rightarrow X+Y=-2.$$
We therefore obtain the system of linear equations
$$\left\{\begin{array}{ccc}
X-Y&=&0\\
X+Y&=&-2
\end{array}\right.$$
of solutions $X=-1$ and $Y=-1$. Thus, 3..\\
By taking $\frac{n+2}{2}<t\leq n$, then letting $X'=\sum_{j=1}^{\frac{n-2}{2}}d_{t-(\frac{n-6}{2})\;j}$ and $Y'=\sum_{j=\frac{n-2}{2}+1}^{n-2}d_{t-(\frac{n-6}{2})\;j}$, then using the same reasoning, we obtain the system of equations
$$\left\{\begin{array}{ccc}
X'-Y'&=&-2\\
X'+Y'&=&0
\end{array}\right.$$
of solutions $X'=-1$ and $Y'=1$. Whence, 4..
\end{proof}
\noindent We describe a procedure that inputs a $BH(m,n)$ and outputs a $BH(m,(\frac{n}{2}-1)n)$. This process can be identified by a mapping
$$\Psi:\mathcal{A}^{(1,2)}(m,n)\to\mathcal{BH}(m,(\frac{n}{2}-1)n)$$
$$\;\;\;\;H\to\Psi(H).$$
We let the sub-matrix $T$ of Lemma 2 be as follows,
$$T=\left[\begin{array}{cc}
T^{[1]}&T^{[2]}
\end{array}\right]=\left[\begin{array}{c}
C\\
D
\end{array}\right]=\left[\begin{array}{cc}
C^{[1]}&C^{[2]}\\
D^{[1]}&D^{[2]}
\end{array}\right]$$
where, $T^{[1]}$ defines the first $\frac{n}{2}-1$ columns of $T$ and $T^{[2]}$ the remaining columns. Similarly for $C^{[1]}$ with $C^{[2]}$ for $C$, and $D^{[1]}$ with $D^{[2]}$ for $D$.\\
Next, we summarize the following theorem.
\begin{Theorem}[Construction theorem]
Let $H\in \mathcal{A}^{(1,2)}(m,n)$. Let $\mathbb{X}=\{\mathcal{X}^{(1)},...,\mathcal{X}^{(\frac{n}{2}-2)}\}$ be a complete tensor set of LSESC of order $\frac{n}{2}-1$. By $G_2$, we associate to $\mathbb{X}$ the tensors set\\ $\{\mathcal{Y}^{(1)}=G_2(\mathcal{X}^{(1)}),\mathcal{Y}^{(2)}=G_2(\mathcal{X}^{(2)}),...,\mathcal{Y}^{(\frac{n}{2}-2)}=G_2(\mathcal{X}^{(\frac{n}{2}-2)})\}$.  Let\\ $B_0=H'\otimes \mathbf{J}_{\frac{n}{2}-1}$, with $H'$ the sub-matrix of $H$ obtained by deleting the two rows $(x_1\;x_2\;...\;x_n)$ and $(x_1\;-x_2\;...\;x_{n-1}\;-x_n)$, and
\begin{eqnarray}B=\left[\begin{array}{c|cccc}
E_1(x_1,x_2)&T_{x_3,x_4}&T_{x_5,x_6}&...&T_{x_{n-1},x_n}\\
E_2(x_1,x_2)&\mathbf{Y}^{(1)}_{1}T_{x_3,x_4}&\mathbf{Y}^{(1)}_{2}T_{x_5,x_6}&...&\mathbf{Y}^{(1)}_{\frac{n}{2}-1}T_{x_{n-1},x_n}\\
E_3(x_1,x_2)&\mathbf{Y}^{(2)}_{1}T_{x_3,x_4}&\mathbf{Y}^{(2)}_{2}T_{x_5,x_6}&...&\mathbf{Y}^{(2)}_{\frac{n}{2}-1}T_{x_{n-1},x_n}\\
\vdots & \vdots&\vdots &...&\vdots\\
E_{\frac{n}{2}-1}(x_1,x_2)&\mathbf{Y}^{(n-2)}_{1}T_{x_3,x_4}&\mathbf{Y}^{(n-2)}_{2}T_{x_5,x_6}&...&\mathbf{Y}^{(n-2)}_{\frac{n}{2}-1}T_{x_{n-1},x_n}\\
\end{array}
\right].
\end{eqnarray}
where
$$E_i(x_1,x_2)=\left[\begin{array}{c}
\mathbf{J}_{\frac{n}{2}-1}^{\top}\otimes [x_1\mathbf{c_i^{[1]}}\;x_2\mathbf{c_i^{[2]}}]\\
\mathbf{J}_{\frac{n}{2}-1}^{\top}\otimes [x_1\mathbf{d_i^{[1]}}\;x_2\mathbf{d_i^{[2]}}]
\end{array}\right],$$
and
$$T_{x_d,x_{d+1}}=\left[\begin{array}{cc}
D_{x_d}T^{[1]}&D_{x_{d+1}}T^{[2]}\\
\end{array}\right]$$
with $D_{x_d}=Diag_{n-2}(x_d,x_d,...,x_d)$.
Then, the matrix
$$\Psi_{\mathbb{X}}(H)=\left[\begin{array}{c}
   B_0\\
  B\\
  \end{array}
\right]$$
is a Butson matrix of order $n(\frac{n}{2}-1)$ with entries powers of the $m$th root of unity.
\end{Theorem}
\begin{proof}
In the same way as in Theorem 1, we proceed by checking the row orthogonality. We first divide $B$ into $(n-2)\times (\frac{n}{2}-1)n$ block matrices $\Gamma_k$ in similar way to Theorem 1. The rows of $\Gamma_0$ are of the form
\begin{eqnarray}\left[x_1\mathbf{c^{[1]}}_1\;x_2\mathbf{c^{[2]}}_1|\;x_3\mathbf{c^{[1]}}_j\;x_4\mathbf{c^{[2]}}_j\;x_5\mathbf{c^{[1]}}_j\;x_6\mathbf{c^{[2]}}_j\;...\;x_{n-1}\mathbf{c^{[1]}}_j\;x_n\mathbf{c^{[2]}}_j\right]\end{eqnarray}
or
\begin{eqnarray}\left[x_1\mathbf{d^{[1]}}_1\;x_2\mathbf{d^{[2]}}_1|\;x_3\mathbf{d^{[1]}}_j\;x_4\mathbf{d^{[2]}}_j\;x_5\mathbf{d^{[1]}}_j\;x_6\mathbf{d^{[2]}}_j\;...\;x_{n-1}\mathbf{d^{[1]}}_j\;x_n\mathbf{d^{[2]}}_j\right].\end{eqnarray}
While the rows of $\Gamma_k$, when $k\neq 0$ are
\begin{eqnarray}\left[x_1\mathbf{c^{[1]}}(k+1)\;x_2\mathbf{c^{[2]}}(k+1)|\;x_3\mathbf{c^{[1]}}(l^{(k)}_{j\;1})\;x_4\mathbf{c^{[2]}}(l^{(k)}_{j\;1})\;...\;x_{n-1}\mathbf{c^{[1]}}(l^{(k)}_{j\;(\frac{n}{2}-1)})\;x_n\mathbf{c^{[2]}}(l^{(k)}_{j\;(\frac{n}{2}-1)})\right].\end{eqnarray}
or
\begin{eqnarray}\left[x_1\mathbf{d^{[1]}}(k+1)\;x_2\mathbf{d^{[2]}}(k+1)|\;x_3\mathbf{d^{[1]}}(l^{(k)}_{j\;1})\;x_4\mathbf{d^{[2]}}(l^{(k)}_{j\;1})\;...\;x_{n-1}\mathbf{d^{[1]}}(l^{(k)}_{j\;(\frac{n}{2}-1)})\;x_n\mathbf{d^{[2]}}(l^{(k)}_{j\;(\frac{n}{2}-1)})\right].\end{eqnarray}
The proof follows by checking four cases:
\begin{itemize}
\item[1)] Two rows of $B_0$ are orthogonal by the properties of the Kroncker product.
\item[2)] A row of $B_0$ and a row of $B$: Here we differ two cases.
                \begin{itemize}
                \item[i.] A row of $B_0$ with a row combined of rows of $C$: We check using rows of $\Gamma_k$, $k\neq 0$, the other case follows similarly.
                    \begin{eqnarray}
                    \mathbf{b_0}_t\cdot\overline{\mathbf{\Gamma}_k}^{\top}_r&=&h'_{t\;1}\cdot \overline{x_1}\overline{\sum_{j=1}^{\frac{n-2}{2}}c_{k+1\;j}}+h'_{t\;2}\cdot \overline{x_2}\overline{\sum_{j=\frac{n-2}{2}+1}^{n-2}c_{k+1\;j}}\nonumber\\
                    &&+h'_{t\;3}\cdot \overline{x_3}\overline{\sum_{j=1}^{\frac{n-2}{2}}[\mathbf{c}(l^{(k)}_{r\;1})]_j}+h'_{t\;4}\cdot \overline{x_4}\overline{\sum_{j=\frac{n-2}{2}+1}^{n-2}[\mathbf{c}(l^{(k)}_{r\;1})]_j}\nonumber\\
                    &&+...+h'_{t\;n-1}\overline{x_{n-1}}\overline{\sum_{j=1}^{\frac{n-2}{2}}[\mathbf{c}(l^{(k)}_{r\;\frac{n}{2}-1})]_j}+h'_{t\;n}\cdot \overline{x_n}\overline{\sum_{j=\frac{n-2}{2}+1}^{n-2}[\mathbf{c}(l^{(k)}_{r\;\frac{n}{2}-1})]_j}\nonumber
                    \end{eqnarray}
                   From Lemma 2, 3. $\sum_{j=1}^{\frac{n-2}{2}}c_{d\;j}=-1$ and $\sum_{j=\frac{n-2}{2}+1}^{n-2}c_{d\;j}=-1$ for all $0<d\leq \frac{n-2}{2}$. Thus, we obtain
                   \begin{eqnarray}
                   \mathbf{b_0}_t\cdot\overline{\mathbf{\Gamma}_k}^{\top}_r&=&-\sum_{i=1}^{n}h'_{t\;i}\cdot\overline{x_i}=0\nonumber
                   \end{eqnarray}
                   As $\mathbf{h'}_t$ and $(x_1\;x_2\;...\;x_n)$ are two rows of $H$.
                \item[ii.] Now taking a row of $B_0$ and other combined from $D$, we obtain
                     \begin{eqnarray}
                    \mathbf{b_0}_t\cdot\overline{\mathbf{\Gamma}_k}^{\top}_{r+(\frac{n-2}{2})}&=&h'_{t\;1}\cdot \overline{x_1}\overline{\sum_{j=1}^{\frac{n-2}{2}}d_{k+1\;j}}+h'_{t\;2}\cdot \overline{x_2}\overline{\sum_{j=\frac{n-2}{2}+1}^{n-2}d_{k+1\;j}}\nonumber\\
                    &&+h'_{t\;3}\cdot \overline{x_3}\overline{\sum_{j=1}^{\frac{n-2}{2}}[\mathbf{d}(l^{(k)}_{r\;1})]_j}+h'_{t\;4}\cdot \overline{x_4}\overline{\sum_{j=\frac{n-2}{2}+1}^{n-2}[\mathbf{d}(l^{(k)}_{r\;1})]_j}\nonumber\\
                    &&+...+h'_{t\;n-1}\overline{x_{n-1}}\overline{\sum_{j=1}^{\frac{n-2}{2}}[\mathbf{d}(l^{(k)}_{r\;\frac{n}{2}-1})]_j}+h'_{t\;n}\cdot \overline{x_n}\overline{\sum_{j=\frac{n-2}{2}+1}^{n-2}[\mathbf{d}(l^{(k)}_{r\;\frac{n}{2}-1})]_j}\nonumber
                    \end{eqnarray}
                    and by Lemma 2, 4.  $\sum_{j=1}^{\frac{n-2}{2}}d_{l\;j}=-1$ and $\sum_{j=\frac{n-2}{2}+1}^{n-2}d_{l\;j}=1$ for all $0<l\leq \frac{n-2}{2}$. Thus
                    \begin{eqnarray}
                   \mathbf{b_0}_t\cdot\overline{\mathbf{\Gamma}_k}^{\top}_{r+(\frac{n-2}{2})}&=&-\sum_{i=1}^{n}h'_{t\;i}\cdot(-1)^{i-1}\cdot\overline{x_i}=0\nonumber
                   \end{eqnarray}
                   as both $\mathbf{h'}_t$ and $(x_1\;-x_2\;...\;x_{n-1}\;-x_n)$ are rows of $H$.
                \end{itemize}
\item[3)] Now we check if two rows of $\Gamma_k$ are orthogonal. We have three possible combinations. First, a row combined from $C$, and the second combined from $D$:
    \begin{eqnarray}
    \mathbf{\Gamma_k}_t\cdot{\overline{\mathbf{\Gamma_k}_{r+\frac{n-2}{2}}}}^{\top}&=&x_1\overline{x_1}\mathbf{c^{[1]}}(k+1)\overline{\mathbf{d^{[1]}}(k+1)}^{\top}+x_2\overline{x_2}\mathbf{c^{[2]}}(k+1)\overline{\mathbf{d^{[2]}}(k+1)}^{\top}\nonumber\\
    &&+x_{3}\overline{x_{3}}\mathbf{c^{[1]}}(l^{(k)}_{t\;1})\overline{\mathbf{d^{[1]}}(l^{(k)}_{r\;1})}^{\top}+x_{4}\overline{x_{ 4}}\mathbf{c^{[2]}}(l^{(k)}_{t\;1})\overline{\mathbf{d^{[2]}}(l^{(k)}_{r\;1})}^{\top}\nonumber\\
    &&+...+x_{n-1}\overline{x_{n-1}}\mathbf{c^{[1]}}(l^{(k)}_{t\;\frac{n}{2}-1})\overline{\mathbf{d^{[1]}}(l^{(k)}_{r\;\frac{n}{2}-1})}^{\top}+x_{n}\overline{x_{n}}\mathbf{c^{[2]}}(l^{(k)}_{t\;\frac{n}{2}-1})\overline{\mathbf{d^{[2]}}(l^{(k)}_{r\;\frac{n}{2}-1})}^{\top}\nonumber\\
    \end{eqnarray}
    $x_l\overline{x_l}=1$ for all $l$, and $$\mathbf{c^{[1]}}(d)\overline{\mathbf{d^{[1]}}(e)}^{\top}+\mathbf{c^{[2]}}(d)\overline{\mathbf{d^{[2]}}(e)}^{\top}=\mathbf{c}_d\cdot\overline{\mathbf{d}}_e^{\top}=0.$$
    For any given appropriate $d$ and $e$. Using Lemma 2 2., this sum concludes to zero.\\
    Secondly, we check two rows combined from $C$, the process is the same if we choose two rows combined from $D$. We have
    \begin{eqnarray}
    \mathbf{\Gamma_k}_t\cdot\overline{\mathbf{\Gamma_k}}_{r}^{\top}&=&x_1\overline{x_1}\mathbf{c^{[1]}}(k+1)\overline{\mathbf{c^{[1]}}(k+1)}^{\top}+x_2\overline{x_2}\mathbf{c^{[2]}}(k+1)\overline{\mathbf{c^{[2]}}(k+1)}^{\top}\nonumber\\
    &&+x_{3}\overline{x_{3}}\mathbf{c^{[1]}}(l^{(k)}_{t\;1})\overline{\mathbf{c^{[1]}}(l^{(k)}_{r\;1})}^{\top}+x_{4}\overline{x_{ 4}}\mathbf{c^{[2]}}(l^{(k)}_{t\;1})\overline{\mathbf{c^{[2]}}(l^{(k)}_{r\;1})}^{\top}\nonumber\\
    &&+...+x_{n-1}\overline{x_{n-1}}\mathbf{c^{[1]}}(l^{(k)}_{t\;\frac{n}{2}-1})\overline{\mathbf{c^{[1]}}(l^{(k)}_{r\;\frac{n}{2}-1})}^{\top}+x_{n}\overline{x_{n}}\mathbf{c^{[2]}}(l^{(k)}_{t\;\frac{n}{2}-1})\overline{\mathbf{c^{[2]}}(l^{(k)}_{r\;\frac{n}{2}-1})}^{\top}\nonumber\\
    &=&\mathbf{c^{[1]}}(k+1)\overline{\mathbf{c^{[1]}}(k+1)}^{\top}+\mathbf{c^{[2]}}(k+1)\overline{\mathbf{c^{[2]}}(k+1)}^{\top}\nonumber\\
    &&+\mathbf{c^{[1]}}(l^{(k)}_{t\;1})\overline{\mathbf{c^{[1]}}(l^{(k)}_{r\;1})}^{\top}+\mathbf{c^{[2]}}(l^{(k)}_{t\;1})\overline{\mathbf{c^{[2]}}(l^{(k)}_{r\;1})}^{\top}\nonumber\\
    &&+...+\mathbf{c^{[1]}}(l^{(k)}_{t\;\frac{n}{2}-1})\overline{\mathbf{c^{[1]}}(l^{(k)}_{r\;\frac{n}{2}-1})}^{\top}+\mathbf{c^{[2]}}(l^{(k)}_{t\;\frac{n}{2}-1})\overline{\mathbf{c^{[2]}}(l^{(k)}_{r\;\frac{n}{2}-1})}^{\top}\nonumber\\
    &=&\mathbf{c}(k+1)\overline{\mathbf{c}(k+1)}^{\top}+\mathbf{c}(l^{(k)}_{t\;1})\overline{\mathbf{c}(l^{(k)}_{r\;1})}^{\top}+...+\mathbf{c}(l^{(k)}_{t\;\frac{n}{2}-1})\overline{\mathbf{c}(l^{(k)}_{r\;\frac{n}{2}-1})}^{\top}\nonumber\\
    &=&(n-2)-2-...-2=(n-2)-2(\frac{n-2}{2})=n-2-(n-2)=0.\nonumber
    \end{eqnarray}
    Here always $l^{(k)}_{t\;i}\neq l^{(k)}_{r\;i}$ as they are entries in the same column in the Latin square.
\item[4)] Finally, by choosing a row from $\Gamma_k$ and another from $\Gamma_s$, we differ three cases. The first is one of these rows is combined from $C$, and the other from $D$. In this case, we will get an equation similar to (4.10) that leads to zero. The second, is if we choose the two rows combined from $C$. So,
    \begin{eqnarray}
    \mathbf{\Gamma_k}_t\cdot\overline{\mathbf{\Gamma_s}}_{r}^{\top}&=&x_1\overline{x_1}\mathbf{c^{[1]}}(k+1)\overline{\mathbf{c^{[1]}}(s+1)}^{\top}+x_2\overline{x_2}\mathbf{c^{[2]}}(k+1)\overline{\mathbf{c^{[2]}}(s+1)}^{\top}\nonumber\\
    &&+x_{3}\overline{x_{3}}\mathbf{c^{[1]}}(l^{(k)}_{t\;1})\overline{\mathbf{c^{[1]}}(l^{(s)}_{r\;1})}^{\top}+x_{4}\overline{x_{ 4}}\mathbf{c^{[2]}}(l^{(k)}_{t\;1})\overline{\mathbf{c^{[2]}}(l^{(s)}_{r\;1})}^{\top}\nonumber\\
    &&+...+x_{n-1}\overline{x_{n-1}}\mathbf{c^{[1]}}(l^{(k)}_{t\;\frac{n}{2}-1})\overline{\mathbf{c^{[1]}}(l^{(s)}_{r\;\frac{n}{2}-1})}^{\top}+x_{n}\overline{x_{n}}\mathbf{c^{[2]}}(l^{(k)}_{t\;\frac{n}{2}-1})\overline{\mathbf{c^{[2]}}(l^{(s)}_{r\;\frac{n}{2}-1})}^{\top}\nonumber\\
    &=&\mathbf{c}(k+1)\overline{\mathbf{c}(s+1)}^{\top}+\mathbf{c}(l^{(k)}_{t\;1})\overline{\mathbf{c}(l^{(s)}_{r\;1})}^{\top}+...+\mathbf{c}(l^{(k)}_{t\;\frac{n}{2}-1})\overline{\mathbf{c}(l^{(s)}_{r\;\frac{n}{2}-1})}^{\top}\nonumber\\
    &=&-2-2-...+(n-2)-2-...-2=n-2-(n-2)=0.\nonumber\\
    \end{eqnarray}
    $l^{(k)}_{t\;i}=l^{(s)}_{r\;i}$ in exactly one position by definition of a LSESC. The third case, if we choose two rows combined from $D$, the result follows immediately by the same process as of $C$.\\
    By taking one of the blocks $\Gamma$ as $\Gamma_0$, we obtain equations similar to (4.11), which ultimately reduce to zero. This outcome occurs because, by the definition of a Latin square, there exists an index $i$ such that $l^{(k)}_{t\;i}=r$.
\end{itemize}
This finishes the proof and shows that $\Psi_{\mathbb{X}}(H)$ is a $BH(m,(\frac{n}{2}-1)n)$.
\end{proof}
\noindent By Remark 2, we can only choose matrices from $\mathcal{A}^{(2)}(m,n)$ as input, and before the definition of $B_0$ we perform column permutations if necessary.
\begin{Example}
 Considering the 6th root of unity $\omega=e^{\dfrac{2\pi\mathbf{i}}{6}}$. Then $\omega^2=\mathbf{j}$, $\omega^3=-1$, $\omega^4=-\omega$ and $\omega^5=-\mathbf{j}$.
 $$F_6=\left[\begin{array}{cccccc}
 1&1&1&1&1&1\\
 1&\omega&\mathbf{j}&-1&-\omega&-\mathbf{j}\\
 1&\mathbf{j}&-\omega&1&\mathbf{j}&-\omega\\
 1&-1&1&-1&1&-1\\
 1&-\omega&\mathbf{j}&1&-\omega&\mathbf{j}\\
 1&-\mathbf{j}&-\omega&-1&\mathbf{j}&\omega
 \end{array}\right].$$
 The sub-matrix $T$ of this matrix is
$$T=\left[\begin{array}{cccc}
-\omega&\mathbf{j}&\mathbf{j}&-\omega\\
\mathbf{j}&-\omega&-\omega&\mathbf{j}\\
\mathbf{j}&-\omega&\omega&-\mathbf{j}\\
-\omega&\mathbf{j}&-\mathbf{j}&\omega
 \end{array}\right].$$
We need LSESC of order $2$ that was defined in the previous example. Next, we compute the lateral slices of $\mathcal{Y}$, yielding the following results.
$$Y_1=\left[\begin{array}{cccc}
1&0&0&0\\
0&1&0&0\\
0&0&1&0\\
0&0&0&1
 \end{array}\right], Y_2=\left[\begin{array}{cccc}
0&1&0&0\\
1&0&0&0\\
0&0&0&1\\
0&0&1&0
 \end{array}\right].$$
 By deleting the first and the fourth rows of $F_6$ we obtain $H'$. The resulting matrix is
$$\Psi_{\mathbb{X}}(F_6)=\left[\begin{array}{cccccccccccc}
1&1&\omega&\omega&\mathbf{j}&\mathbf{j}&-1&-1&-\omega&-\omega&-\mathbf{j}&-\mathbf{j}\\
1&1&\mathbf{j}&\mathbf{j}&-\omega&-\omega&1&1&\mathbf{j}&\mathbf{j}&-\omega&-\omega\\
1&1&-\omega&-\omega&\mathbf{j}&\mathbf{j}&1&1&-\omega&-\omega&\mathbf{j}&\mathbf{j}\\
1&1&-\mathbf{j}&-\mathbf{j}&-\omega&-\omega&-1&-1&\mathbf{j}&\mathbf{j}&\omega&\omega\\
-\omega&\mathbf{j}&\mathbf{j}&-\omega&-\omega&\mathbf{j}&\mathbf{j}&-\omega&-\omega&\mathbf{j}&\mathbf{j}&-\omega\\
-\omega&\mathbf{j}&\mathbf{j}&-\omega&\mathbf{j}&-\omega&-\omega&\mathbf{j}&\mathbf{j}&-\omega&-\omega&\mathbf{j}\\
\mathbf{j}&-\omega&\omega&-\mathbf{j}&\mathbf{j}&-\omega&\omega&-\mathbf{j}&\mathbf{j}&-\omega&\omega&-\mathbf{j}\\
\mathbf{j}&-\omega&\omega&-\mathbf{j}&-\omega&\mathbf{j}&-\mathbf{j}&\omega&-\omega&\mathbf{j}&-\mathbf{j}&\omega\\
\mathbf{j}&-\omega&-\omega&\mathbf{j}&-\omega&\mathbf{j}&\mathbf{j}&-\omega&\mathbf{j}&-\omega&-\omega&\mathbf{j}\\
\mathbf{j}&-\omega&-\omega&\mathbf{j}&\mathbf{j}&-\omega&-\omega&\mathbf{j}&-\omega&\mathbf{j}&\mathbf{j}&-\omega\\
-\omega&\mathbf{j}&-\mathbf{j}&\omega&\mathbf{j}&-\omega&\omega&-\mathbf{j}&-\omega&\mathbf{j}&-\mathbf{j}&\omega\\
-\omega&\mathbf{j}&-\mathbf{j}&\omega&-\omega&\mathbf{j}&-\mathbf{j}&\omega&\mathbf{j}&-\omega&\omega&-\mathbf{j}
\end{array}\right].$$
\end{Example}
\noindent We have for any power of prime number an example of a complete set of LSESC, the set $\mathcal{L}$. Then using Lemma 1 2. we conclude the following.
\begin{Corollary}
For any $r\in\mathbb{N}^*$, there exists $BH(2(2^r+1),2^{r+1}(2^r+1))$ obtained by the operator $\Psi$.
\end{Corollary}
\noindent This construction can be extended to include more matrices that might be inequivalent, by inputting two Butson matrices of order $n$ on the same root of unity. This is done by defining the map
$$\Psi:\mathcal{A}^{(1)}(m,n)\times \mathcal{A}^{(2)}(m,n)\to\mathcal{BH}(m,(\frac{n}{2}-1)n)$$
$$\;\;\;\;(G,H)\to\Psi(G,H).$$
We notice that in this way we extend the number of eligible input matrices, as, for instance all the Fourier matrices can be imputed as the first matrix. We notice also that $\mathcal{A}^{(1,2)}(m,n)=\mathcal{A}^{(1)}(m,n)\cap \mathcal{A}^{(2)}(m,n)$.
\begin{Corollary}
Let $(G,H)\in\mathcal{A}^{(1)}(m,n)\times \mathcal{A}^{(2)}(m,n)$. Under the same assumptions on LSESC as in theorem 2, we get
$$\Psi_{\mathbb{X}}(G,H)=\left[\begin{array}{c}
B_0\\
B
\end{array}\right],$$
with $B_0=G'\otimes \mathbf{J}_{\frac{n}{2}-1}$, where $G'$ the matrix obtained from $G$ by deleting its $(x_1\;x_2\;...\;x_n)$ and $(x_1\;-x_2\;...\;x_{n-1}\;-x_n)$ rows, and $B$ defined as in (4.5), by taking $T$ the sub-matrix of $H$ obtained as shown in Lemma 2.
\end{Corollary}
\begin{proof}
The proof follows exactly in the same way as in Theorem 2. The orthogonality of rows of $B_0$ is saved by the Kronecker product. The orthogonality of rows of $B$ follows exactly as shown in Theorem 2. The proof of orthogonality of a row of $B$ and other from $B_0$ take the same steps as mentioned in 2) of the Theorem 2, by replacing the row $\mathbf{h'}_t$ by $\mathbf{g'}_t$.
\end{proof}
\noindent By applying this procedure to Hadamard matrices, we obtain a construction of this matrices. Note that any Hadamard matrix verifies \textbf{C1} and \textbf{C2}. The first is because any normalized Hadamard matrix contains a row $(1\;-1\;1\;...\;-1)$, and the second is trivial. So, we define a process on Hadamard matrices:
$$\Psi:\mathcal{H}_n\times\mathcal{H}_n\to\mathcal{H}_{(\frac{n}{2}-1)n}$$
$$\;\;\;\;(G,H)\to\Psi(G,H).$$
\begin{Corollary}
Let $(G,H)\in\mathcal{H}_n\times\mathcal{H}_n$. Under the same assumptions as in Theorem 2 for the set of LSESC, we obtain the matrix
$$\Psi_{\mathbb{X}}(G,H)=\left[\begin{array}{c}
B_0\\
B
\end{array}\right]$$
with $B_0=G'\otimes \mathbf{J}_{\frac{n}{2}-1}$, where $G'$ the matrix obtained from $G$ by deleting its $(x_1\;x_2\;...\;x_n)$ and $(x_1\;-x_2\;...\;x_{n-1}\;-x_n)$ rows, and $B$ defined as in (4.5), by taking $T$ the sub-matrix of $H$ similarly as shown in Lemma 2.
\end{Corollary}
\begin{proof}
The proof follows exactly as in Theorem 2.
\end{proof}
\noindent We conclude this section by counting all the possible outputs of $\Psi$ for a given $n$. We denote by $d_H$ the half of the number of rows that verify \textbf{C1} in an input matrix $H$. Then the number of obtained matrices is
$$\sum_{H\in\mathcal{A}^{(1)}(m,n)}|MOLS(\frac{n}{2}-1)|\cdot|\mathcal{A}^{(2)}(m,n)|\cdot d_H.$$
Similarly, for Hadamard matrices, we obtain
$$\sum_{H\in\mathcal{H}_n}|MOLS(\frac{n}{2}-1)|\cdot|\mathcal{H}_n|\cdot d_H$$
different matrix. These matrices can be equivalent, and their classification remains as an open problem.
\section{Conclusion}
We proposed a method for constructing Butson matrices of order $n(n-1)$ from existing Butson matrices of order $n$. This construction can be achieved by using either one input matrix or two input matrices, thereby increasing the number of resulting matrices, which may be inequivalent. Additionally, we described a similar approach for constructing Butson matrices of order $n(\frac{n}{2}-1)$ from one or two existing matrices of order $n$. The input matrices must satisfy specific conditions, and we demonstrated that some Fourier matrices meet these requirements. Furthermore, we derived results pertaining to the construction of Hadamard matrices, which enhance our previous findings presented in \cite{h2}.

\end{document}